\documentclass[reqno,12pt]{amsart}%
\usepackage{amsfonts}
\usepackage{amsmath}
\usepackage{amssymb}

\def\bE{{\mathbb{E}}}

\def\bR{{\mathbb{R}}}

\def\cJ{{\mathcal{J}}}

\def\ba{\boldsymbol{\alpha}}
\def\bep{\boldsymbol{\epsilon}}

\def\mfu{\mathfrak{u}}
\def\Hep{{\mathrm{H}}}

\newcommand\mfk[1]{\mathfrak{#1}}

\def\ba{\boldsymbol{\alpha}}

\def\mfm{\mathfrak{m}}

\newtheorem{theorem}{Theorem}
\theoremstyle{plain}
\newtheorem{corollary}[theorem]{Corollary}

\newtheorem{proposition}[theorem]{Proposition}

\newcommand\bel[1]{\begin{equation}\label{#1}}
\newcommand\ee{\end{equation}}

\numberwithin{theorem}{section}
\numberwithin{equation}{section}

\begin{document}
\title[Stationary PAM]
{An Asymptotic Comparison of Two Time-Homogeneous PAM Models}
\author{Hyun-Jung Kim}
\curraddr[H.-J. Kim]{Department of Mathematics, USC\\
Los Angeles, CA 90089}
\email[H.-J. Kim]{hyunjungmath@gmail.com}
\urladdr{https://hyunjungkim.org/}
\author{Sergey Vladimir Lototsky}
\curraddr[S. V. Lototsky]{Department of Mathematics, USC\\
Los Angeles, CA 90089}
\email[S. V. Lototsky]{lototsky@math.usc.edu}
\urladdr{http://www-bcf.usc.edu/$\sim$lototsky}



\begin{abstract}
Both Wick-It\^{o}-Skorokhod and Stratonovich
 interpretations of the Parabolic Anderson model (PAM) lead to solutions that
 are real analytic as functions of the noise intensity $\varepsilon$, and,
  in the limit $\varepsilon\to 0$, the difference between the
  two solutions is of order $\varepsilon^2$ and is non-random.
\end{abstract}

\maketitle


\section{Introduction}

Let $W=W(x),\ x\in [0,\pi]$ be a standard Brownian motion on a
complete probability space $(\Omega, \mathcal{F},\mathbb{P})$.
With no loss of generality, we assume that all realizations  of $W$
are in $\mathcal{C}^{1/2-}((0,\pi))$, that is,
H\"{o}lder continuous of every order less than $1/2$.

Consider the equations
\bel{eq:main-W}
\begin{split}
\frac{\partial u_{\diamond}(t,x;\varepsilon)}{\partial t} &=
\frac{\partial^2 u_{\diamond}(t,x;\varepsilon)}{\partial x^2} +
\varepsilon u_{\diamond}(t,x;\varepsilon)\diamond \dot{W}(x),\ t>0, \
0<x<\pi,\\
u_{\diamond}(t,0;\varepsilon)&=u_{\diamond}(t,\pi;\varepsilon)=0, \
u_{\diamond}(0,x;\varepsilon)=\varphi(x),
\end{split}
\ee
and
\bel{eq:main-S}
\begin{split}
\frac{\partial u_{\circ}(t,x;\varepsilon)}{\partial t} &=
\frac{\partial^2 u_{\circ}(t,x;\varepsilon)}{\partial x^2} +
\varepsilon u_{\circ}(t,x;\varepsilon)\circ \dot{W}(x),\ t>0, \
0<x<\pi,\\
u_{\circ}(t,0;\varepsilon)&=u_{\circ}(t,\pi;\varepsilon)=0, \
u_{\circ}(0,x;\varepsilon)=\varphi(x).
\end{split}
\ee
Equation \eqref{eq:main-W} is the Wick-It\^{o}-Skorokhod formulation of the
parabolic Anderson model with potential $\varepsilon\dot{W}$; equation \eqref{eq:main-S}
is the corresponding Stratonovich (or geometric rough path) formulation.
These equations, with $\varepsilon=1$,
 are studied in \cite{Kim-Lot17} and \cite{Kim-Lot-GRP},
respectively.

The objective of the paper is to show that
\begin{itemize}
\item The solutions of \eqref{eq:main-W} and
\eqref{eq:main-S} are real-analytic functions of
$\varepsilon$: with suitable functions $u_{\diamond}^{(n)}$,
and $u_{\circ}^{(n)}$, the equalities
\begin{align}
\label{ps-W}
u_{\diamond}(t,x;\varepsilon)&=u_{\diamond}(t,x;0)+
\sum_{n=1}^{\infty} \varepsilon^n u_{\diamond}^{(n)}(t,x)\\
\label{ps-S}
u_{\circ}(t,x;\varepsilon)&=u_{\circ}(t,x;0)+
\sum_{n=1}^{\infty} \varepsilon^n u_{\circ}^{(n)}(t,x)
\end{align}
hold for all $t>0$, $x\in [0,\pi]$, $\varepsilon>0$, and every
realization of $W$;
\item The first two terms in \eqref{ps-W} and \eqref{ps-S} are the
same so that
\bel{order2}
|u_{\diamond}(t,x;\varepsilon)-u_{\circ}(t,x;\varepsilon)|=O(\varepsilon^2),
 \ \varepsilon\to 0,
\ee
 for all $t>0$ and $x\in [0,\pi]$,  and every
realization of $W$.
\end{itemize}
Equalities \eqref{ps-W} and \eqref{ps-S} are in the spirit of \cite{Lot2002}.
Equality \eqref{order2} is similar to \cite[Proposition 4.1]{Wan-Wick12};
see also \cite{Wan-Wick10}.

The precise statement of the main result is in Section \ref{sec:MR}, and
the proof is in Sections \ref{sec:W}, \ref{sec:S}, and \ref{sec:Cr}.

\section{The Main Result}
\label{sec:MR}

Denote by $p=p(t,x,y)$ the heat semigroup on $[0,\pi]$ with
zero boundary conditions:
\bel{HS}
p(t,x,y)=\frac{2}{\pi}\sum_{k=1}^{\infty} e^{-k^2t}\sin(kx)\, \sin(ky),\
t>0,\ x,y\in [0,\pi].
\ee
Let $\varphi=\varphi(x)$ be a continuous function on $[0,\pi]$,
 and let $u=u(t,x)$ be the solution of the heat equation
\bel{u0}
\begin{split}
\frac{\partial u(t,x)}{\partial t} &=
\frac{\partial^2 u(t,x)}{\partial x^2},\ t>0, \
0<x<\pi,\\
u(t,0)&=u(t,\pi)=0, \
u(0,x)=\varphi(x),
\end{split}
\ee
that is,
\bel{u0-expl}
u(t,x)=\int_0^{\pi}p(t,x,y)\varphi(y)dy.
\ee
Next, define the function $\mfu=\mfu(t,x)$ by
\bel{u1-expl}
\mfu(t,x)=\int_0^t\int_0^{\pi} p(t-s,x,y)u(s,y)\,dW(y)\, ds.
\ee
That is, $\mfu$ is the mild solution of
\bel{u1}
\begin{split}
\frac{\partial \mfu(t,x)}{\partial t} &=
\frac{\partial^2 \mfu(t,x)}{\partial x^2}+u(t,x)\dot{W}(x),\ t>0, \
0<x<\pi,\\
\mfu(t,0)&=\mfu(t,\pi)=\mfu(0,x)=0.
\end{split}
\ee
Because $u$ is non-random, no stochastic integral is required to define $\mfu$.

\begin{proposition}
If $\varphi\in \mathcal{C}((0,\pi))$, then $\mfu$ is a continuous function of
$t$ and $x$ for all $t>0$ and $x\in [0,\pi]$.
\end{proposition}

\begin{proof}
This follows by the Kolmogorov continuity criterion: $\mfu$ is a Gaussian
random field and direct computations show
$$
\bE \big(\mfu(t+\tau,x+h)-\mfu(t,x)\big)^2\leq
C(t)\big(\tau^2+h^2\big)^{1/4}\,
\max_{x\in [0,\pi]}|\varphi(x)|;
$$
cf. \cite[Sections 6 and 7]{Kim-Lot17}.
\end{proof}

Next, define the functions $u_{\diamond}^{(n)}=u_{\diamond}^{(n)}(t,x)$,
$n=0,1,2\ldots$, $t\geq 0$, $x\in [0,\pi]$, by $u_{\diamond}^{(0)}(t,x)=
u(t,x)$, and, for $n\geq 1$, $u_{\diamond}^{(n)}$ is the mild
solution of
\bel{eq:derivative-W}
\begin{split}
\frac{\partial {u}^{(n)}_{\diamond}(t,x)}{\partial t} &=
\frac{\partial^2 {u}^{(n)}_{\diamond}(t,x)}{\partial x^2} +
 {u}^{(n-1)}_{\diamond}(t,x)\diamond \dot{W}(x),\ t>0, \ 0<x<\pi,\\
{u}^{(n)}_{\diamond}(t,0)&={u}^{(n)}_{\diamond}(t,\pi)
= {u}^{(n)}_{\diamond}(0,x)=0.
\end{split}
\ee
In other words,
\bel{eq:derivative-W-expl}
u_{\diamond}^{(n)}(t,x)=\int_0^t \int_0^{\pi} p(t-s,x,y)
u_{\diamond}^{(n-1)}(s,y)\diamond d{W}(y) \,  ds,\ \ n\geq 1,
\ee
and, in particular, $u_{\diamond}^{(1)}=\mfu$.

Similarly, define the functions $u_{\circ}^{(n)}=u_{\circ}^{(n)}(t,x)$,
$n=0,1,2\ldots$, $t\geq 0$, $x\in [0,\pi]$, by $u_{\circ}^{(0)}(t,x)=
u(t,x)$, and, for $n\geq 1$, $u_{\circ}^{(n)}$ is the mild
solution of
\bel{eq:derivative-S}
\begin{split}
\frac{\partial {u}^{(n)}_{\circ}(t,x)}{\partial t} &=
\frac{\partial^2 {u}^{(n)}_{\circ}(t,x)}{\partial x^2} +
 {u}^{(n-1)}_{\circ}(t,x)\circ \dot{W}(x),\ t>0, \ 0<x<\pi,\\
{u}^{(n)}_{\circ}(t,0)&={u}^{(n)}_{\circ}(t,\pi)
= {u}^{(n)}_{\circ}(0,x)=0.
\end{split}
\ee
In other words,
\bel{eq:derivative-S-expl}
u_{\circ}^{(n)}(t,x)=\int_0^t \int_0^{\pi} p(t-s,x,y)
u_{\circ}^{(n-1)}(s,y)\circ d{W}(y) \,  ds,\ \ n\geq 1,
\ee
and, in particular, $u_{\circ}^{(1)}=\mfu$.

The main result of the paper can now be stated as follows.

\begin{theorem}
\label{th:main}
 Let $\varphi\in \mathcal{C}((0,\pi))$. Then
\begin{enumerate}
\item Equality \eqref{ps-W} holds with
$u_{\diamond}^{(n)}$ from \eqref{eq:derivative-W-expl}.
\item Equality \eqref{ps-S} holds with
$u_{\circ}^{(n)}$ from \eqref{eq:derivative-S-expl}.
\item Equality \eqref{order2} holds and
\bel{order2-d}
\lim_{\varepsilon\to 0}\varepsilon^{-2}
\Big(u_{\circ}(t,x;\varepsilon)-u_{\diamond}(t,x;\varepsilon)\Big)=
\int_0^{\pi} p^{(3)}(t,x,z)\varphi(z)dz,
\ee
where
$$
p^{(3)}(t,x,z)=\int_0^{\pi}\int_0^t\int_0^s
p(t-s,x,y)p(s-r,y,y)p(r,y,z)\,dr\,ds\,dy.
$$
\end{enumerate}
\end{theorem}
The proof is carried out in the following three sections.

\section{The Wick-It\^{o}-Skorokhod Case}
\label{sec:W}

The objective of this section is the proof of \eqref{ps-W}.

The  solution of \eqref{eq:main-W} is defined as a chaos solution
(cf. \cite[Theorems 3.10]{LR-spn}). It
 is a continuous in $(t,x)$ function
 (cf. \cite[Sections 6 and 7]{Kim-Lot17})
 and has a  representation as a series
\bel{W-sol}
u_{\diamond}(t,x;\varepsilon)=
\sum_{\ba\in \cJ}u_{\ba}(t,x;\varepsilon)\xi_{\ba},
\ee
where
\begin{align*}
\cJ&=\left\{\ba=(\alpha_k,\, k\geq 1): \alpha_k\in \{0,1,2,\ldots\},\
\sum_{k}\alpha_k<\infty\right\}, \\
\xi_{\ba}& = \prod_{k}
 \left(
 \frac{\Hep_{\alpha_{k}}(\xi_{k})}{\sqrt{\alpha_{k}!}}\right),\
\Hep_{n}(x) = (-1)^{n} e^{x^{2}/2}\frac{d^{n}}{dx^{n}}%
e^{-x^{2}/2},\\
\xi_k&=\int_0^{\pi}\mfk{m}_k(x)\,dW(x),\
\mfk{m}_k(x)=\sqrt{2/\pi}\,\sin(kx),
\end{align*}
and, with $|\ba|=\sum_k\alpha_k$, $|(\boldsymbol{0})|=0$,
$ |\bep(k)|=1$,
\begin{align}
\notag
u_{(\boldsymbol{0})}(t,x;\varepsilon)&=u(t,x),\\
\notag
u_{\ba}(t,x;\varepsilon)&=
\varepsilon \sum_{k} \sqrt{\alpha_k}\,
\int\limits_0^t\int\limits_0^{\pi} p(t-s,x,y)
u_{\ba-\bep(k)}(s,y;\varepsilon)\mfk{m}_k(y)\,dyds;
\end{align}
see \cite[Section 3]{Kim-Lot17} for details.
In particular,
\bel{uba-ep}
\sum_{|\ba|=n}u_{\ba}(t,x;1)\xi_{\ba}=
\sum_{|\ba|=n-1}\int_0^t\int_0^{\pi}
p(t-s,x,y)u_{\ba}(s,y;1)\xi_{\ba}\diamond dW(y)\, ds.
\ee
 Comparing \eqref{eq:derivative-W-expl} with  \eqref{uba-ep}
 shows that
\bel{n-der-W}
u_{\diamond}^{(n)}(t,x)=\sum_{|\ba|=n}u_{\ba}(t,x;1)\xi_{\ba}.
\ee
In other words, \eqref{ps-W} is equivalent  to \eqref{W-sol}.

Next,
\begin{equation*}
\begin{split}
\bE|u_{\diamond}^{(n)}(t,x)|=&
\bE\left|\sum_{|\ba|=n}u_{\ba}(t,x;1)\xi_{\ba}\right|
\leq \left(\bE\left(\sum_{|\ba|=n}
u_{\ba}(t,x;1)\xi_{\ba}\right)^2\right)^{1/2}\\
&= \left(\sum_{|\ba|=n}|u_{\ba}(t,x;1)|^2\right)^{1/2}
\leq C^n(t) {n^{-n/4}}\, \sup_{x\in (0,\pi)}|\varphi(x)|^{1/2},
\end{split}
\end{equation*}
where the last inequality follows by \cite[Theorem 4.1]{Kim-Lot17}.
As a result,
$$
\sum_{n\geq 0}\varepsilon^n \bE \big|u_{\diamond}^{(n)}(t,x)\big|<\infty,
$$
that is,  the series converges absolutely with probability one for all
$t>0,\ x\in [0,\pi]$, and $\varepsilon \in \bR$.

{\em This concludes the proof of  \eqref{ps-W}.}

\section{The Stratonovich Case}
\label{sec:S}

The objective of this section is to prove  \eqref{ps-S}.
To simplify the presentation,  we use the following notations:
 \begin{align*}
\Lambda&=(-\boldsymbol{\Delta})^{1/2}, \ \ H^{\theta}=
\Lambda^{-\theta}\big(L_2((0,\pi)\big),\ \
\|\cdot\|_{\theta}=\|\Lambda^{\theta}\cdot\|_{L_2((0,\pi))},\ \
\theta\in \bR,\\
 \ p*g(t,s,x)&=\int_0^{\pi} p(t-s,x,y)g(s,y)\,dy,
\end{align*}
where $\boldsymbol{\Delta}$ is the Laplace operator on $(0,\pi)$ with
zero boundary conditions and  $p$ is the heat kernel \eqref{HS}.

By direct computation,
\bel{PosNorm}
\|p*g\|_{\gamma}(t,s)\leq C_{T,\theta}\,
(t-s)^{-\theta/2}\|g\|_{\gamma-\theta}(s),\ \theta>0,\ \
\gamma\in \bR,\ \ t\in (s,T];
\ee
cf. \cite[Lemma 7.3]{Kr_Lp1}.

Consider the equation
\bel{ghe-1}
\frac{\partial v(t,x)}{\partial t}=\frac{\partial^2 v(t,x)}{\partial x^2}+
v(t,x)\circ \dot{W}(x)+f(t,x)\circ \dot{W}(x),\ t>0,
\ee
which includes \eqref{eq:main-S} as a particular case. By definition,
a solution (classical, mild, generalized, etc.) of  \eqref{ghe-1} is a suitable
limit, as $\epsilon\to 0$,  of the corresponding solutions of
\bel{ghe-1-ep}
\frac{\partial v_{\epsilon}(t,x)}{\partial t}
=\frac{\partial^2 v_{\epsilon}(t,x)}{\partial x^2}+
v_{\epsilon}(t,x)V_{\epsilon}(x)+f(t,x)V_{\epsilon}(x) ,\ t>0,
\ee
where $V_{\epsilon}, \ \epsilon>0$ are smooth functions on $[0,\pi]$
such that
$$
\sup_{\epsilon}
\left\| \int V_{\epsilon}\right\|_{\mathcal{C}^{1/2}}<\infty,\
\lim_{\epsilon\to 0} \sup_{x\in [0,\pi]}
\left| \int_0^xV_{\epsilon}(y)dy-W(x)\right|=0.
$$

By \cite[Theorem 3.5]{Kim-Lot-GRP},
\begin{itemize}
\item The generalized solution of \eqref{ghe-1} is the
same as the generalized solution of the equation
\bel{Str-gen}
v_t=\Big(v_{x}+W(x)v+W(x)f\Big)_x-W(x)v_x-W(x)f_x;
\ee
the subscripts $t$ and $x$, as in $f_x$,  denote the corresponding
partial derivatives;
\item The mild solution of \eqref{ghe-1} is the solution of the
integral equation
 \bel{Str-mild}
 \begin{split}
v(t,x)&= \int_0^t p*\big((f+v)W\big)_x(t,s,x)\, ds-
\int_0^t p*\big((f+v)_xW\big)(t,s,x)\,
ds\\
&+\int_0^{\pi}p(t,x,y)\varphi(y)dy.
\end{split}
\ee
\end{itemize}

On the one hand,   mild and generalized  solutions of \eqref{ghe-1}
are the same: just use $\mfk{m}_k$ as the test functions. On the other
hand, different definitions of the solution lead to different regularity results.

 By standard parabolic regularity, if
$\varphi\in H^{0}$ and $f\in L_2\big((0,T);H^{\gamma}\big)$,
$\gamma\in (1/2,1]$,
 then there is a unique generalized solution of \eqref{Str-gen}
 in the normal triple $(H^1,H^0,H^{-1})$ and
\bel{Str-gen-reg}
v\in  L_2\big((0,T);H^{1}\big)\bigcap
\mathcal{C}\big((0,T);H^{0}\big)
\ee
for every realization of $W$; cf. \cite[Theorem 3.4.1]{LM}.
Note that we cannot claim
$
v\in \mathcal{C}\big((0,T);H^{\gamma}\big)
$
even if $\varphi\in H^{\gamma}$.
 In fact,
because $W\in \mathcal{C}^{1/2-}$ is a point-wise multiplier in
$H^{\gamma}$ for $\gamma\in (-1/2,1/2)$
\cite[Lemma 5.2]{Kr_Lp1}, an attempt to find
 a traditional  regularity result for equation
\eqref{Str-gen} in a normal triple $(H^{r+1}, H^r, H^{r-1})$
leads to an irreconcilable pair of restrictions on $r$:
to have $Wf\in L_2((0,T); H^{r})$ we need $r<1/2$, whereas
to have $Wf_x\in L_2((0,T); H^{r-1})$ we need $r-1>-1/2$ or
$r>1/2$.

Accordingly, to derive a bound on $\|v\|_{\gamma}(t)$ for $t>0$,
 we use the mild formulation \eqref{Str-mild}.

\begin{proposition}
\label{pr:BasicEstimate-Str}
Let $\gamma\in (1/2,1)$, $f\in L_2\big((0,T); H^{\gamma}\big)$,
$\varphi\in H^{0}$,  and
let $v$ be the mild solution of \eqref{ghe-1} with
$v|_{t=0}=\varphi$.
Then, for every $T>0$ and every realization of $W$, there exists
a number $C_{\circ}$ such that
\bel{eq:stab0}
\|v\|_{\gamma}(t) \leq C_{\circ}\left(t^{-\gamma/2}\|\varphi\|_{0}+
\int_0^t (t-s)^{-\gamma}\|f\|_{\gamma}(s)\, ds\right).
\ee

\end{proposition}

\begin{proof}
Throughout the proof, $C$ denotes a number depending
on $\gamma$, $T$, and the
 norm of $W$ in the space $\mathcal{C}^{1-\gamma}$.
 The value of $C$ can change from one instance to another.
With no loss of generality, we assume that $\varphi$ and
$f$ are smooth functions with compact support.

To begin, let us show that if $V$ is the mild solution of
$$
\frac{\partial V(t,x)}{\partial t}=\frac{\partial^2 V(t,x)}{\partial x^2}+
f(t,x)\circ \dot{W}(x),\ t>0,
$$
$V(t,0)=V(t,\pi)=0$, $V|_{t=0}=\varphi$,
then
\bel{eq:stab1}
\|V\|_{\gamma}(t) \leq Ct^{-\gamma/2}\|\varphi\|_{0}+C
\int_0^t (t-s)^{-\gamma}\|f\|_{\gamma}(s)\, ds.
\ee
Indeed, by \eqref{Str-mild},
$$
V(t,x)= \int_0^t p*(fW)_x(t,s,x)\, ds-
\int_0^t p*\big(f_xW\big)(t,s,x)\, ds+\int_0^{\pi}p(t,x,y)\varphi(y)dy.
$$
Using  \eqref{PosNorm}  with $\theta=\gamma$,
$$
\left\|\int_0^{\pi}p(t,\cdot,y)\varphi(y)dy\right\|_{\gamma}\leq Ct^{-\gamma/2}\|\varphi\|_0.
$$
Then
\bel{str-sol-f}
\|V\|_{\gamma}(t)\leq \int_0^t \|p*(fW)_x\|_{\gamma}(t,s)\, ds+
\int_0^t \|p*\big(f_xW\big)\|_{\gamma}(t,s)\, ds
+Ct^{-\gamma/2}\|\varphi\|_{0}.
\ee

To estimate the first term on the right hand side of \eqref{str-sol-f},
 we use  \eqref{PosNorm}  with $\theta=2\gamma$. Then
$$
\|p*(fW)_x\|_{\gamma}(t,s)\leq
C(t-s)^{-\gamma}\|(fW)_x\|_{-\gamma}(s)
\leq C(t-s)^{-\gamma}\|fW\|_{1-\gamma}(s),
$$
and, because $W\in \mathcal{C}^{1/2-}((0,\pi))$
 is a (point-wise) multiplier
in  $H^{1-\gamma}$,
$$
\|fW\|_{1-\gamma}(s)\leq C_W\|f\|_{1-\gamma}(s);
$$
 recall that $0<1-\gamma<1/2$.
Finally, as $1-\gamma<\gamma$,
\bel{str-est1-f}
\|p*(fW)_x\|_{\gamma}(t,s)\leq
C(t-s)^{-\gamma}\|f\|_{\gamma}(s).
\ee

To estimate the second term on the right hand side of \eqref{str-sol-f},
 we use  \eqref{PosNorm}  with $\theta=1$. Then
$$
\|p*\big(f_xW)\|_{\gamma}(t,s)
\leq \frac{C}{\sqrt{t-s}}\,\|f_xW\|_{\gamma-1}(s)
\leq \frac{C}{\sqrt{t-s}}\,\|f_x\|_{\gamma-1}(s),
$$
that is,
\bel{str-est2-f}
\|p*\big(f_xW\big)\|_{\gamma}(t,s)
\leq C (t-s)^{-1/2}\|f\|_{\gamma}(s).
\ee
To establish \eqref{eq:stab1}, we now combine \eqref{str-sol-f},
 \eqref{str-est1-f}, and \eqref{str-est2-f},
  keeping in mind that
 $(t-s)^{-1/2}\leq C(t-s)^{-\gamma}$ because $\gamma>1/2$.

Next,   \eqref{eq:stab1} applied to \eqref{ghe-1} implies
$$
\|v\|_{\gamma}(t)\leq C\int_0^t (t-s)^{-\gamma}
\|v\|_{\gamma}(s)\,ds+C\int_0^t (t-s)^{-\gamma}
\|f\|_{\gamma}(s)\,ds+Ct^{-\gamma/2}\|\varphi\|_{0},
$$
and then a generalization  of the Gronwall inequality
 (e.g. \cite[Corollary 2]{GG}) leads to \eqref{eq:stab0}.
\end{proof}

\begin{corollary}
If  $\varphi\in H^{0}$ and $\gamma\in (1/2,1)$, then, for
every $T>0$, $a>0$, and every realization of $W$, there
 exists a number $\tilde{C}_{\circ}$ such that the mild
solution of \eqref{eq:main-S} satisfies
\bel{norm-u}
\sup_{|\varepsilon|<a}\|u_{\circ}(t,\cdot,\varepsilon)\|_{\gamma}
\leq \tilde{C}_{\circ}\,t^{-\gamma/2}\|\varphi\|_0,\ t\in (0,T].
\ee
\end{corollary}

Next, define the functions $u_{\circ}^{(n),\varepsilon}=
u_{\circ}^{(n),\varepsilon}(t,x)$,\ $n=0,1,2\ldots,$
$t\geq 0$, $x\in [0,1]$, $\varepsilon\in \bR$,
by $u_{\circ}^{(0),\varepsilon}(t,x)=u_{\circ}(t,x;\varepsilon)$ and, for
$n\geq 1$,
\bel{eq:main-dk-ep}
\begin{split}
\frac{\partial u_{\circ}^{(n),\varepsilon}(t,x)}{\partial t} &=
\frac{\partial^2 u_{\circ}^{(n),\varepsilon}(t,x)}{\partial x^2} +
\varepsilon u_{\circ}^{(n),\varepsilon}(t,x)\circ \dot{W}(x)+
 u_{\circ}^{(n-1),\varepsilon}(t,x)\circ \dot{W}(x),\\
u_{\circ}^{(n),\varepsilon}(t,0)
&=u_{\circ}^{(n),\varepsilon}(t,\pi)=0, \
u_{\circ}^{(n),\varepsilon}(0,x)=0.
\end{split}
\ee
 In particular,
 $$
 u_{\circ}^{(n),0}(t,x)=u_{\circ}^{(n)}(t,x).
 $$
 Note that all equations in \eqref{eq:main-dk-ep}
  are of the form \eqref{ghe-1}.

\begin{proposition}
\label{prop:St}
If $\varphi\in H^{0}$, then, for every $\gamma\in (1/2,2/3)$
 and every realization of $W$,
\bel{gen-der}
\lim_{\varepsilon\to \varepsilon_0}
\frac{1}{(\varepsilon-\varepsilon_0)^{n}}
\Big\|u_{\circ}(t,\cdot;\varepsilon)-\sum_{k=0}^n
(\varepsilon-\varepsilon_0)^ku_{\circ}^{(k),\varepsilon_0}(t,\cdot)
\Big\|_{\gamma}=0,
\ee
$n\geq 0$, $\varepsilon_0\in \bR$,\ $t\geq 0$.
\end{proposition}

\begin{proof}
Throughout the proof, $C$ denotes a number depending
on $\gamma$, $T$, $\varepsilon_0$, and the
 norm of $W$ in the space $\mathcal{C}^{1-\gamma}$.
Define
\bel{dffrence-gen}
\mathfrak{v}^{(n)}_{\varepsilon}(t,x)=
\frac{1}{(\varepsilon-\varepsilon_0)^{n}}
\left(u_{\circ}(t,x;\varepsilon)-\sum_{k=0}^n
(\varepsilon-\varepsilon_0)^ku_{\circ}^{(k),\varepsilon_0}(t,x)\right).
\ee
By \eqref{eq:main-S},
\bel{eq:main-Vd1}
\begin{split}
\sum_{k=0}^n(\varepsilon-\varepsilon_0)^k
\frac{\partial u_{\circ}^{(k),\varepsilon_0}(t,x)}{\partial t}
&= \sum_{k=0}^n
(\varepsilon-\varepsilon_0)^k
\frac{\partial^2 u_{\circ}^{(k),\varepsilon_0}(t,x)}{\partial x^2}\\
&+\varepsilon_0\sum_{k=0}^n(\varepsilon-\varepsilon_0)^k
u_{\circ}^{(k),\varepsilon_0}(t,x)\circ \dot{W}(x)\\
 &+ (\varepsilon-\varepsilon_0)\sum_{k=1}^n
 (\varepsilon-\varepsilon_0)^{k-1}
  u_{\circ}^{(k-1),\varepsilon_0}(t,x) \circ \dot{W}(x),
\end{split}
\ee
so that
\bel{eq:main-Vd2-gen}
\begin{split}
\frac{\partial \mathfrak{v}^{(0)}_{\varepsilon}(t,x)}{\partial t} &=
\frac{\partial^2 \mathfrak{v}^{(0)}_{\varepsilon}(t,x)}{\partial x^2}
+\varepsilon_0\mathfrak{v}^{(0)}_{\varepsilon}(t,x)
\circ \dot{W}(x)
 + (\varepsilon-\varepsilon_0) u_{\circ}(t,x;\varepsilon)\circ
 \dot{W}(x),\\
\frac{\partial \mathfrak{v}^{(n)}_{\varepsilon}(t,x)}{\partial t} &=
\frac{\partial^2 \mathfrak{v}^{(n)}_{\varepsilon}(t,x)}{\partial x^2}
+\varepsilon_0\mathfrak{v}^{(n)}_{\varepsilon}(t,x)\circ \dot{W}(x)
 +  \mathfrak{v}^{(n-1)}_{\varepsilon}(t,x)\circ \dot{W}(x),
 \ n\geq 1,
\end{split}
\ee
 $\mathfrak{v}^{(n)}_{\varepsilon}(0,x)=0$, $n\geq 0$, and
\eqref{gen-der} becomes
\bel{gen-der-v}
\lim_{\varepsilon\to\varepsilon_0}
\|\mathfrak{v}^{(n)}_{\varepsilon}\|_{\gamma}(t)=0,\
\ n\geq 0,\ t\geq 0,\ \varepsilon_0\in \bR.
\ee
 Note that all equations in
\eqref{eq:main-Vd2-gen} are of the form \eqref{ghe-1},
and \eqref{gen-der-v} trivially holds for $t=0$.
Accordingly, combining the second equation in \eqref{eq:main-Vd2-gen}
with \eqref{eq:stab0},
$$
\|\mathfrak{v}^{(n)}_{\varepsilon}\|_{\gamma}(t) \leq
C\int_0^t
(t-s)^{-\gamma}
\|\mathfrak{v}^{(n-1)}_{\varepsilon}\|_{\gamma}(s)\,ds,
$$
$n\geq 1$,
and then, for $t>0$, \eqref{gen-der-v} follows by induction: for  $n=0$,
 \eqref{norm-u} yields
\begin{align*}
&\|\mathfrak{v}^{(0)}_{\varepsilon}\|_{\gamma}(t)
\leq |\varepsilon-\varepsilon_0|\,
C\int_0^t (t-s)^{-\gamma}\|u_{\circ}(s,\cdot,\varepsilon)\|_{\gamma}\,ds
\\
& \leq  C  |\varepsilon-\varepsilon_0|\, \|\varphi\|_0
\int_0^t (t-s)^{-\gamma}s^{-\gamma/2}\,ds
\leq C  |\varepsilon-\varepsilon_0|\, \|\varphi\|_0 t^{1-(3/2)\gamma}
\to 0,\ \varepsilon \to \varepsilon_0;
\end{align*}
similarly,  for $n\geq 1$,
$$
\|\mathfrak{v}^{(n)}_{\varepsilon}\|_{\gamma}(t) \leq
C^{(n)}|\varepsilon-\varepsilon_0|\, \|\varphi\|_0,
$$
because $1-(3/2)\gamma>0$.
\end{proof}

\begin{proposition}
\label{prop:Sa}
If $\varphi\in H^{0}$ and  $\gamma\in (1/2,1)$, then
 \bel{n-der-S}
\lim_{n\to \infty} c^n\sup_{|\varepsilon|<a}
\|u_{\circ}^{(n),\varepsilon}\|_{\gamma}(t)=0,\ t\geq 0,
\ee
for all $c>0$, $a>0$, and every realization of $W$.
\end{proposition}

\begin{proof}
Throughout this proof,  $C$ denotes a number depending
on $\gamma$, $T$, $a$, and the
 norm of $W$ in the space $\mathcal{C}^{1-\gamma}$.

Combining  \eqref{eq:main-dk-ep} and \eqref{eq:stab0},
$$
\|u_{\circ}^{(n),\varepsilon}]\|_{\gamma}(t) \leq
C\int_0^t (t-s)^{r-1}\|u_{\circ}^{(n-1),\varepsilon}\|_{\gamma}(s)\,ds.
$$
 By iteration and \eqref{norm-u}, with $r=1-\gamma>0$,
\begin{equation*}
\begin{split}
\sup_{|\varepsilon|<a}
&\|u_{\circ}^{(n),\varepsilon}\|_{\gamma}(t)
\leq C^n\|\varphi\|_0\\
&\times\int_0^t\int_0^{s_{n-1}}\!\!\!\!\!\ldots\int_0^{s_2}
(t-s_n)^{r-1}(s_n-s_{n-1})^{r-1}
\cdots (s_2-s_1)^{r-1}s_1^{-\gamma/2}ds_1\cdots ds_n\\
&=C^n \|\varphi\|_0 \,
 \frac{\big(\Gamma(r)\big)^n\Gamma(1-(\gamma/2))}
 {\Gamma\big(nr+1\big)} \,t^{nr-(\gamma/2)},
\end{split}
\end{equation*}
 where $\Gamma$  is the Gamma function
$$
\Gamma(y)=\int_0^{\infty} t^{y-1}e^{-t}dt.
$$
Then \eqref{n-der-S} follows by the Stirling formula.
\end{proof}

Equality  \eqref{ps-S} now follows:
\begin{itemize}
\item By the Sobolev embedding theorem, every element,
 or equivalence, class from $H^{\gamma}$, $\gamma>1/2$,
  has a representative
 that is a continuous function on $[0,\pi]$;
\item By Proposition \ref{prop:St} and the Taylor formula,
$$
u_{\circ}(t,x)=u(t,x)+\sum_{k=1}^n
u_{\circ}^{(k),0}(t,x)\varepsilon^k+R_n(t,x);
$$
\item By Proposition \ref{prop:Sa},
$$
\lim_{n\to \infty} R_n(t,x)=0.
$$
\end{itemize}

{\em This concludes the proof of  \eqref{ps-S}.}

\section{The Correction Term}
\label{sec:Cr}

The objective of this section is the proof of \eqref{order2-d}.

Using \eqref{ps-W} and \eqref{ps-S}, and remembering that
$u_{\circ}^{(1)}=u_{\diamond}^{(1)}=\mathfrak{u}$,
\begin{equation}
\label{WS-diff0}
\begin{split}
\lim_{\varepsilon\to 0}\varepsilon^{-2}&
\Big(u_{\circ}(t,x;\varepsilon)-u_{\diamond}(t,x;\varepsilon)\Big)=
u_{\circ}^{(2)}(t,x)-u_{\diamond}^{(2)}(t,x)\\
&=
\int_0^t \int_0^{\pi} p(t-s,x,y)\big(\mathfrak{u}(s,y)\circ d{W}(y)-
\mathfrak{u}(s,y)\diamond d{W}(y)\big)\,ds.
\end{split}
\end{equation}
By definition,
$$
\xi_k\diamond\xi_n=
\begin{cases}
\xi_k\xi_n,& k\not=n,\\
\xi_n^2-1,& k=n,
\end{cases}
$$
and therefore
\bel{WDif}
\xi_k\xi_n-\xi_k\diamond\xi_n=
\begin{cases}
0& k\not=n,\\
1,& k=n,
\end{cases}
\ee
Then  \eqref{WDif} and \cite[Theorem 3.1.2]{Nualart} imply that,
for  a function $f=f(x)$ of the form
 $$
 f(x)=\sum_{k=1}^{\infty} f_k(x)\xi_k,
 $$
 with $f_k$ non-random and satisfying
 \bel{int-cond}
 \sum_k\int_0^{\pi}|f_k(x)|\,dx <\infty,
 \ee
  the following equality holds:
 \bel{WS-dif1}
 \int_0^{\pi} f(x)\circ dW(x)-\int_0^{\pi}
 f(x)\diamond dW(x)=
 \sum_{k=1}^{\infty} \left(\int_0^{\pi} f_k(x)\mathfrak{m}_k(x)\,
 dx\right).
 \ee
 Condition \eqref{int-cond} ensures that the sum on the
 right-hand side of \eqref{WS-dif1} converges absolutely.

 Next, recall that, by \eqref{u1-expl},
$$
\mathfrak{u}(s,y)=\sum_{k=1}^{\infty}
\left(\int_0^{\pi} \int_0^s p(s-r,y,z)u(r,z)\mathfrak{m}_k(z)\, dr\, dz
\right)\,\xi_k.
$$
For fixed $s\in [0,T]$ and $y\in[0,\pi]$, define
$$
g(z)=\int_0^s p(s-r,y,z)u(r,z)dr,\ g_k=\int_0^{\pi}g(z)\mfm_k(z)\,dz.
$$
Then
$$
\mathfrak{u}=\sum_{k=1}^{\infty} g_k\xi_k,
$$
and \eqref{int-cond} in this case will follow from uniform, 
in $(s,y)$ convergence of 
$$
\sum_{k=1}^{\infty} |g_k|,
$$
which, by Bernstein's theorem \cite[Theorem VI.3-1]{Zygmund},
will, in turn,  follow from 
\bel{HoldCont}
|g(z+h)-g(z)|\leq Ch^{\delta}
\ee
with $\delta\in (1/2,1)$ and $C$ independent of $s,y,z$. 

Recall that 
$$
u(r,z)=\sum_{k=1}^{\infty} \varphi_k e^{-k^2r} \, \mfm_k(z),\ 
\varphi_k=\int_0^{\pi} \varphi(x)\mfm_k(x)\, dx,
$$
and, by integral comparison, 
$$
\sum_{k=1}^{\infty} k^p e^{-k^2\,t}\leq \frac{C(p)}{t^{(1+p)/2}},\
p\geq 0.
$$
Also,
$$
|\sin(k(z+h))-\sin(kz)|\leq k^\delta h^{\delta},\ \delta\in (0,1),
$$
and the maximum principle implies $|u(r,z)|\leq C$.
Then
\begin{align*}
&p(s-r,y,z)\leq \frac{C}{\sqrt{s-r}},\ \
|p(s-r,y,z+h)-p(s-r,y,z)|\leq \frac{Ch^{\delta}}{(s-r)^{(1+\delta)/2}},\\
& |u(r,z+h)-u(r,z)|\leq \frac{Ch^{\delta}}{r^{(1+\delta)/2}},
\end{align*}
and \eqref{HoldCont} follows because
$$
\int_0^s\frac{dr}{\big(r(s-r)\big)^{(1+\delta)/2}}<\infty
$$
for $\delta\in (1/2,1)$.

We  now apply \eqref{WS-dif1} to \eqref{WS-diff0}:
 \begin{align*}
&\int_0^t \int_0^{\pi} p(t-s,x,y)\big(\mathfrak{u}(s,y)\circ d{W}(y)-
\mathfrak{u}(s,y)\diamond d{W}(y)\big)\,ds\\
&=\int_0^t \int_0^{\pi} \sum_{n=1}^{\infty}
\left(\int_0^{\pi} \!\left(\int_0^s p(s-r,y,z) {u}(r,z)\,dr\right)
\mathfrak{m}_n(z)\, dz\right)\mathfrak{m}_n(y)p(t-s,x,y)\, dy\,ds\\
&= \int_0^{\pi}\int_0^t\int_0^s p(t-s,x,y)p(s-r,y,y)u(r,y)\, dr\, ds\, dy,
\end{align*}
which, in view of \eqref{u0-expl} and the Fubini theorem, is the same as
\eqref{order2-d}. 

{\em This concludes the proof of  \eqref{order2-d}.}


\begin{thebibliography}{10}

\bibitem{Kim-Lot17}
Kim, H.-J. and Lototsky, S.~V.: Time-homogeneous parabolic {W}ick-{A}nderson model in one space
  dimension: regularity of solution, \emph{Stoch. Partial Differ. Equ. Anal. Comput.}
  \textbf{5} (2017), no.~4, 559--591.

\bibitem{Kim-Lot-GRP}
Kim, H.-J. and  Lototsky, S.~V.: Heat equation with a geometric rough path
  potential in one space dimension: {E}xistence and regularity of solution,
  https://arxiv.org/abs/1712.08196.

\bibitem{Kr_Lp1}
Krylov, N.~V.: \emph{An analytic approach to {SPDEs}, Stochastic Partial
  Differential Equations}, Six Perspectives, Mathematical Surveys and Monographs
  (B.~L. Rozovsky and R.~Carmona, eds.), AMS, 1999.

\bibitem{LM}
Lions, J.-L. and Magenes, E.: \emph{Probl\'{e}mes aux limites non homog\`{e}nes
  et applications, vol. {I}}, Dunod, Paris, 1968.

\bibitem{Lot2002}
Lototsky, S.~V.: Small perturbation of 
stochastic parabolic equations: a
  power series analysis, \emph{J. Funct. Anal.} \textbf{193} (2002), no.~1, 94--115.

\bibitem{LR-spn}
Lototsky, S.~V. and Rozovsky, B.~L.: Stochastic partial differential
  equations driven by purely spatial noise, \emph{SIAM J. Math. Anal.} \textbf{41}
  (2009), no.~4, 1295--1322.

\bibitem{Nualart}
Nualart, D.: \emph{Malliavin calculus and related topics, 2nd ed.}, Springer,
  New York, 2006.

\bibitem{Wan-Wick10}
Wan, X.: A note on stochastic elliptic models, \emph{Comput. Method. Appl. M.}
  \textbf{199} (2010), no.~45-48, 2987--2995.

\bibitem{Wan-Wick12}
Wan, X.: A discussion on two stochastic elliptic modeling strategies,
  \emph{Commun. Comput. Phys.} \textbf{11} (2012), 775--796.

\bibitem{GG}
Ye, H., Gao, J., and Ding, Y.: A generalized {G}ronwall inequality and its
  application to a fractional differential equation, \emph{J. Math. Anal. Appl.}
  \textbf{328} (2007), no.~2, 1075--1081.

\bibitem{Zygmund}
Zygmund, A.: \emph{Trigonometric series, vol. {I}, 3rd ed.}, Cambridge
  Mathematical Library, Cambridge University Press, 2002.

\end{thebibliography}

\def\cprime{$'$}
\providecommand{\bysame}{\leavevmode\hbox to3em{\hrulefill}\thinspace}
\providecommand{\MR}{\relax\ifhmode\unskip\space\fi MR }
\providecommand{\MRhref}[2]{%
  \href{http://www.ams.org/mathscinet-getitem?mr=#1}{#2}
}
\providecommand{\href}[2]{#2}

\end{document}